\documentclass[11pt]{amsart}%
\usepackage{amsmath}
\usepackage{graphicx}
\usepackage{amsfonts}
\usepackage{amssymb}
\usepackage{hyperref}%
\providecommand{\U}[1]{\protect\rule{.1in}{.1in}}
\newtheorem{theorem}{Theorem}[section]
\theoremstyle{plain}

\newtheorem{corollary}{Corollary}[section]

\newtheorem{lemma}{Lemma}[section]

\numberwithin{equation}{section}

\def\re{\mathbb R}

\DeclareMathOperator{\sym}{Sym}
\DeclareMathOperator{\diag}{diag} 
\DeclareMathOperator{\spn}{span}
\begin{document}
\title[Constant rank theorem]{A constant rank theorem for special Lagrangian Equations\\}
\author{W. Jacob Ogden}
\address{Department of Mathematics, Box 354350\\
University of Washington\\
Seattle, WA 98195}
\email{wjogden@uw.edu}
\author{Yu YUAN}
\address{Department of Mathematics, Box 354350\\
University of Washington\\
Seattle, WA 98195}
\email{yuan@math.washington.edu}
\date{May 28, 2024\\
\indent 2020\emph{ Mathematics subject classification.} 35J60, 35B08, 35B50.}

\begin{abstract}
Constant rank theorems are obtained for saddle solutions to the special Lagrangian equation
and the quadratic Hessian equation. The argument also leads to Liouville type
results for the special Lagrangian equation with subcritical phase, matching the known
rigidity results for semiconvex entire solutions to the quadratic Hessian
equation.

\end{abstract}
\maketitle

\section{\bigskip Introduction}

In this note, we establish constant rank results for the special Lagrangian equation%

\begin{equation}
\sum_{i=1}^{n}\arctan\lambda_{i}=\Theta, \label{EsLag}%
\end{equation}
and the  quadratic Hessian equation on the positive branch
\begin{equation}
\sigma_{2}\left(  D^{2}u\right)  =\sum_{1\leq1<j\leq n}\lambda_{i}\lambda
_{j}=1\ \text{with }\lambda_{1}+\cdots+\lambda_{n}>0, \label{Esigma2}%
\end{equation}
where $\lambda_{1}\leq\cdots\leq\lambda_{n}$ are the eigenvalues of the Hessian $D^{2}u$ and $\Theta$ is constant.

A constant rank theorem asserts that the minimum or maximum
eigenvalue of the Hessian of a solution to a certain elliptic equation obeys
the strong minimum or maximum principle. In other words,  the Hessian shifted by the corresponding scalar matrix, $D^{2}u-\lambda_{\min} I$ or $\lambda_{\max} I-D^{2}u$, has
constant rank. That certain condition for fully nonlinear elliptic equations is the inverse-convexity 
of Alvarez-Lasry-Lions in \cite{ALL}, employed to obtain the constant rank
result for those various equations by Caffarelli-Guan-Ma in \cite{CGM}.

By connecting the convexity of the level sets of equation \eqref{EsLag} at
critical or supercritical phase $\left\vert \Theta\right\vert \geq\left(
n-2\right)  \pi/2$ \cite{Y2}, to the inverse-convexity of \eqref{EsLag} at
subcritical phase $\left\vert \Theta\right\vert <\left(  n-2\right)  \pi/2,$
we obtain the following constant rank theorem for the special Lagrangian equation
\eqref{EsLag}, which is new for saddle solutions.

\begin{theorem}\label{CrankSlag}
If $u$ is a smooth solution of \eqref{EsLag} in a (connected) domain with
$\arctan\lambda_{\min}\geq\left(  \Theta-\pi\right)  /n$ and $\lambda_{\min}$
attains a local minimum at an interior point, then $\lambda_{\min}$ is
constant. Furthermore, if in addition the local minimum $\arctan\lambda_{\min
}>\left(  \Theta-\pi\right)  /n,$ then there is a fixed direction $e$ such
that the Hessian $D^{2}u$ splits as $\lambda_{\min}\equiv u_{ee}\equiv
const.$  If $\lambda_{\max}$ attains a local interior maximum, 
the corresponding conclusions for $\lambda_{\max}$ also hold, 
provided  $\arctan\lambda_{\max}\leq\left(
\Theta+\pi\right)  /n$ and $\arctan\lambda_{\max}<\left(  \Theta+\pi\right)
/n$ respectively.
\end{theorem}

One application of Theorem \ref{CrankSlag} is the following rigidity result for homogeneous degree 2 solutions.

\begin{corollary} \label{2homogeneous}
If $u$ is a smooth degree 2 homogeneous solution of \eqref{EsLag} on $\re^n \setminus \{0\},$ and $\arctan\lambda_{\min}>\left(  \Theta-\pi\right)  /n$ or
$\arctan\lambda_{\max}<\left(  \Theta+\pi\right)  /n,$ then $u$ is a quadratic polynomial.
\end{corollary}

As a more general by-product of the proof of Theorem \ref{CrankSlag}, we find a Liouville type result for entire solutions of
\eqref{EsLag}, which is new at subcritical and critical phases $\left\vert
\Theta\right\vert \leq\left(  n-2\right)  \pi/2$ in dimension $n\geq5.$

\begin{theorem} \label{rigidity}
If $u$ is an entire smooth solution of \eqref{EsLag} on $\mathbb{R}^{n}$ with
$\arctan\lambda_{\min}\geq\gamma>\left(  \Theta-\pi\right)  /n$ or
$\arctan\lambda_{\max}\leq\gamma<\left(  \Theta+\pi\right)  /n,$ then $u$ is a
quadratic polynomial.
\end{theorem}

Relying on the same inverse-convexity used to derive the Liouville type result for almost convex entire solutions to \eqref{Esigma2} in \cite{CY}, we have the following
constant rank result for the quadratic Hessian equation \eqref{Esigma2}.

\begin{theorem}\label{Cranksigma2}
If $u$ is a smooth solution of \eqref{Esigma2} in a (connected) domain with
$\lambda_{\min}\geq-\sqrt{2/\left[  n\left(  n-1\right)  \right]  }$ and
$\lambda_{\min}$ attains a local minimum at interior point, then
$\lambda_{\min}$ is constant. Furthermore, if in addition $\lambda_{\min
}>-\sqrt{2/\left[  n\left(  n-1\right)  \right]  },$ then there is a fixed
direction $e$ such that the Hessian $D^{2}u$ splits as $\lambda_{\min}\equiv
u_{ee}\equiv const.$
\end{theorem}

A sufficient condition for fully nonlinear elliptic equations to have a lower 
constant rank theorem is convexity. But \eqref{EsLag} with convex solutions 
and the equivalent equation of \eqref{Esigma2}, $\ \Lambda\left(
D^{2}u\right)  =\sqrt{\sigma_{2}\left(  D^{2}u\right)  }=1,$ are concave, not convex.
A relaxed convexity, as alluded to earlier, the inverse-convexity of a
fully nonlinear elliptic equation $G\left(  D^{2}u\right)  =c,$ that is,
convexity of $G\left(  A^{-1}\right)  $ in terms of $A,$ permits a strong
minimum principle for the minimum eigenvalue of the Hessian of solutions. In
terms of the Legendre transform $u^{\ast}=w\left(  y\right)  $ of strongly
convex solution $u\left(  x\right)  $ to $G\left(  D^{2}u\right)  =c,$ as%
\[
\left(  x,Du\left(  x\right)  \right)  =\left(  Dw\left(  y\right)  ,y\right)
\ \text{and then }D^{2}u=\left(  D^{2}w\right)  ^{-1},
\]
the equation $G\left(  D^{2}u\right)  =c$  becomes a concave fully nonlinear elliptic equation%
\[
\Lambda\left(  D^{2}w\right)  =-G\left(  \left(  D^{2}w\right)  ^{-1}\right)
=-c.
\]
Now given $\lambda_{\min}\left(  D^{2}u\right)  $ attains its local minimum,
we have $\lambda_{\max}\left(  D^{2}w\right)  =\lambda_{\min}^{-1}\left(
D^{2}u\right)  $ reaching its local maximum, and say, $\lambda_{\max}\left(
D^{2}w\right)  =w_{11}$ at the local maximum point. By the strong maximum principle,
subsolution $w_{11}\equiv l\equiv$ $\lambda_{\max}\left(  D^{2}w\right)  .$
The gradient graphs split off a line%
\[
\left(  x_{1},x^{\prime},u_{1}(  x)  ,D^{\prime}u(  x)
\right )  =\left(  ly_{1},D^{\prime}w(  y^{\prime})  ,y_{1}%
,y^{\prime}\right)  .
\]
Accordingly, $u_{1}\left(  x\right)  =l^{-1}x_{1},$ hence $u_{11}\equiv
\lambda_{\min}\left(  D^{2}u\right)  $ is constant.

The subtlety lies in the case $\lambda_{\min}\left(  D^{2}u\right)  =0,$ where
the Legendre transformation is not possible. The strong minimum principle for
$\lambda_{\min}$ was achieved in \cite{CGM} by realizing its superharmonicity in a
hidden smooth form of $\bigtriangleup_{G}\lambda_{\min}\leq B\left(  x\right)
\cdot D\lambda_{\min}.$ Moreover, the Hessian $D^{2}u$ may no longer split, that
is, the the minimum-eigendirection may not be fixed, as counterexamples by Korevaar-Lewis
in \cite[p.31]{KL} and also by Sz\'{e}kelyhidi-Weinkove in \cite[p.6530]{SW1} indicate. A quantitative version of the constant rank theorem was discovered in \cite{SW2}.

The lower constant rank results for convex solutions in Theorem \ref{CrankSlag} and
Theorem \ref{Cranksigma2} are contained in \cite[p.1772]{CGM}, as the two corresponding equations
are inverse-convex then. The inverse convexity was also employed in deriving the upper, or equivalently lower, constant rank result for solutions of \eqref{EsLag} with the eigenvalues of the Hessian between -1 and 1  by Bhattacharya-Shankar in \cite[p.18]{BS}.
The connection to the inverse-convexity of \eqref{EsLag} with saddle solutions 
in Theorem \ref{CrankSlag} and \cite{BS} is via the rotation developed in \cite{Y1}, once recognized as 
the Legendre transformation up to translation and scaling.

Note that the lower bound in Theorem \ref{CrankSlag} with $\Theta=\pi/2$ and $n=3$ is
exactly the one $\lambda_{\min}>\tan\left(  -\pi/6\right)  =-1/\sqrt{3}=-m$ in
Theorem \ref{Cranksigma2} with $n=3.$ This is because here the algebraic form of
\eqref{EsLag} is \eqref{Esigma2}. In this instance, the new rotated equation
$\sum_{i=1}^{n}\arctan\bar{\lambda}_{i}=-\pi/2$ is equivalent to $\sigma
_{2}\left(  D^{2}\bar{u}\right)  =1$ on the negative branch $\bar{\lambda}%
_{1}+\bar{\lambda}_{2}+\bar{\lambda}_{3}<0.$ In terms of $\mu_{i}=\left(
\lambda_{i}+m\right)  ^{-1},$ equation $\sigma_{2}\left(  \bar{\lambda
}\right)  =1$ becomes $\Lambda\left(  \mu\right)  =\sigma_{2}\left(
\mu\right)  /\sigma_{1}\left(  \mu\right)  =1/\left(  2m\right)  ,$ which is
concave. In other words, $D^{2}u+mI$ satisfies an inverse-convex equation,
with abused notation, $-\Lambda\left(  \left(  D^{2}u+mI\right)  ^{-1}\right)
=-1/\left(  2m\right)  .$ This particular three dimensional case inspired the
Legendre-Lewy transformation, $(u+m\left\vert x\right\vert ^{2}/2)^{\ast}$ with
$m=\sqrt{2/\left[  n\left(  n-1\right)  \right]  }$, used to derive the Liouville type
result in \cite{CY}, and the constant rank result Theorem \ref{Cranksigma2} for almost convex
solutions of \eqref{Esigma2} with $D^{2}u>-mI$ in general dimensions. The
rigidity result for general semiconvex entire solutions to \eqref{Esigma2}
needed more effort in \cite{SY}.

The above rotation used in proving Theorem \ref{CrankSlag} turns an entire solution to
\eqref{EsLag} with the specific lower or upper Hessian bound into a new entire
one with bounded Hessian on a concave/convex level set, Evans-Krylov's theorem
then leads to the rigidity result Theorem \ref{rigidity}. Coupled with the quadratic
rigidity of 2-homogeneous analytic solutions to a uniformly elliptic Hessian
equation in $\mathbb{R}^{4}\setminus \{  0\}  $ in \cite{NV2}, the argument 
in \cite[p.124--125]{Y1} also shows the Liouville type result for entire solutions to
\eqref{EsLag} with arbitrary lower or upper Hessian bound now in four
dimensions. A Bernstein type result, namely a quadratic rigidity for entire
solutions to \eqref{EsLag} with supercritical phase $\left\vert \Theta
\right\vert >\left(  n-2\right)  \pi/2$ was derived in \cite{Y2}. Precious entire
solutions $W\left(  x\right)  =\left(  x_{1}^{2}+x_{2}^{2}-1\right)  e^{x_{3}%
}+e^{-x_{3}}/4$ and $L\left(  x\right)  =x_{1}^{2}x_{2}-\frac{2}{3}x_{2}%
^{3}-x_{2}x_{3}$ to \eqref{EsLag} with critical $\Theta=\pi/2$ or $\sigma
_{2}\left(  D^{2}W\right)  =1$ and \eqref{EsLag} with subcritical $\Theta=0$
or $\bigtriangleup L=\det D^{2}L$ were constructed by Warren in \cite{W} and by C.-Y. Li in \cite{Lc}, respectively.

Special Lagrangian equation \eqref{EsLag} is the potential equation for
minimal gradient graph or special
Lagrangian graph $\left(  x,Du\left(  x\right)  \right)  \subset\mathbb{R}%
^{n}\times\mathbb{R}^{n}.$ While there are singular $C^{1,\alpha}$ or
Lipschitz viscosity solutions to \eqref{EsLag} with subcritical phase
\cite{NV1, WdY, MS}, regularity for $C^{1,1}$ solutions 
in five dimensions and higher is
an open question, which is equivalent to the quadratic rigidity of
2-homogeneous solutions to \eqref{EsLag} with subcritical phase. Nontrivial
2-homogeneous solutions were found to a saddle uniformly elliptic Hessian
equation in five dimensions \cite{NTV}. Should a nontrivial 2-homogeneous solution
exist to \eqref{EsLag} with subcritical phase, say for example $\Theta=0,$ the
lower bound for the Hessian must be less than or equal to $-\tan\left(
\pi/5\right)  .$ Otherwise, a splitting of the Hessian by Corollary \ref{2homogeneous} implies the triviality of
the supposed example. It would be marvelous that the lower bound of the Hessian is
exactly the $-\tan\left(  \pi/5\right)  $ in Theorem \ref{CrankSlag}.

In the appendix of this paper, we present another proof of a general constant
rank theorem for inverse-convex elliptic Hessian equations, in viscosity way,
besides the smooth \cite{CGM, BG} and approximation \cite{SW1,SW2} ones. 
The argument is to
demonstrate the next smallest nonidentically vanishing eigenvalue is a
supersolution in viscosity sense. In particular, the demonstration of the
constancy of the smallest eigenvalue is simple. After our paper was finished, we found a similar viscosity approach for the constant rank theorem had appeared in the work of Bryan-Ivaki-Scheuer \cite{BIS}.

\section{Proof of Theorem \ref{CrankSlag} and Theorem \ref{rigidity}}

\begin{proof}
[Proof of Theorem \ref{CrankSlag}.]We show the assertions for $\lambda_{\min}.$ These
assertions applied to the solution $v=-u$ of the
equation $\sum_{i=1}^{n}\arctan\lambda_{i}\left(  D^{2}v\right)  =-\Theta$ yield
the corresponding conclusions for $\lambda_{\max}\left(  D^{2}u\right)  .$ We
consider three cases: $\Theta\leq\left(  2-n\right)  \pi/2,$ $\Theta
\in(\left(  2-n\right)  \pi/2,\pi],$ and $\Theta\in\left(  \pi,n\pi/2\right)
.$

Case $\Theta\leq\left(  2-n\right)  \pi/2:$ By \cite[Lemma 2.1]{Y2}, the level set
of equation $\ F\left(  D^{2}u\right)  =\sum_{i=1}^{n}\arctan\lambda_{i}=\Theta$
is concave as a graph along the gradient $\nabla_{\lambda}\sum_{i=1}^{n}\arctan\lambda
_{i},$ then $\bigtriangleup_{F}u_{ee}=F_{ij}\partial_{ij}u_{ee}\leq0.$
Therefore, a strong minimum principle holds for $u_{ee}.$ If $\lambda_{1}$
reaches its local minimum ($>-\infty$), then after change of $x$-coordinates if
necessary, $u_{11}$ reaches a local minimum of the same value of $\lambda
_{1},$ at the same point, so $u_{11}\equiv\lambda_{1}$ is constant.

Case $\Theta\in\left(  \pi,n\pi/2\right)  :$ The hypothesis $\arctan
\lambda_{1}\geq\frac{\Theta-\pi}{n}$ implies $\lambda_{n}\geq\cdots\geq
\lambda_{1}>0,$ so $u$ is strongly convex. Let $u^{\ast}$ be the Legendre
transform of $u.$ The eigenvalues $\mu_{i}$ of $D^{2}u^{\ast}=\left(
D^{2}u\right)  ^{-1}$ satisfy $\mu_{i}=\lambda_{i}^{-1}>0$ and (1.1) becomes%
\[
\sum_{i=1}^{n}\arctan\mu_{i}=n\frac{\pi}{2}-\Theta.
\]
Since $\mu_{i}>0,$ this elliptic equation for $u^{\ast}$ is concave, so we
have a strong maximum principle for pure second derivatives of $u^{\ast}.$ In turn,
this also yields a strong minimum principle for pure second derivatives of $u.$ The
conclusion of Theorem \ref{CrankSlag} follows.

Case $\Theta\in(\left(  2-n\right)  \pi/2,\pi]:$ Let $\alpha=\left(
\Theta-\pi\right)  /n\in(-\pi/2,0]$ and $\beta=\frac{\pi}{2}+\alpha\in
(0,\pi/2].$ First we assume that the attained local minimum is strictly larger than the assumed lower bound, that is%
\[
\frac{\pi}{2}>\theta_{n}\geq\cdots\geq\theta_{1}>\alpha.
\]
where $\theta_i = \arctan \lambda_i. $
As in \cite[p.124--125]{Y1}, we are able to represent the Lagrangian graph $(
x,y)  =(  x,Du(  x)  )  \subset\mathbb{R}%
^{n}\times\mathbb{R}^{n}$ with new anti-clockwise $\beta$-rotated coordinate
axes $\partial_{\bar{x}}+\sqrt{-1}\partial_{\bar{y}}=\left(  c+\sqrt
{-1}s\right)  \left(  \partial_{x}+i\partial_{y}\right),  $
where $c=\cos\beta$ and $s=\sin\beta,$%
\begin{equation}
\left(  cx+sDu\left(  x\right)  ,-sx+cDu\left(  x\right)  \right)  =\left(
\bar{x},D\bar{u}\left(  \bar{x}\right)  \right)  \label{b-roration}%
\end{equation}
such that%
$$
\bar{\theta}_{i}=\theta_{i}-\beta=\theta_{i}-\frac{\pi}{2}-\alpha\in\left(
-\pi/2,\pi/2\right) $$ and 
\begin{equation}
\bar{F}\left(  D^{2}\bar{u}\right)  =\sum_{i=1}^{n}\arctan\bar{\lambda}%
_{i}=\left(  2-n\right)  \frac{\pi}{2}, \label{EsLagC}%
\end{equation}
where $\bar{\theta}_{i}=\arctan\bar{\lambda}_{i}$ and $\bar{\lambda}%
_{i}^{\prime}s$ are the eigenvalues of the Hessian $D^{2}\bar{u}$ of the new
potential $\bar{u}.$

Again by \cite[Lemma 2.1]{Y2}, the level set of the equation \eqref{EsLagC} is
concave as a graph along the gradient $\nabla_{\bar{\lambda}}\sum_{i=1}^{n}\arctan
\bar{\lambda}_{i},$ then $\bigtriangleup_{\bar{F}}\bar{u}_{\bar{e}\bar{e}%
}=\bar{F}_{ij}\partial_{ij}\bar{u}_{\bar{e}\bar{e}}\leq0.$ Thus a strong
minimum principle holds for $\bar{u}_{\bar{e}\bar{e}}.$

Now $\lambda_{\min}=\lambda_{1}$ reaches its local minimum, which is strictly
larger than $\tan\left(  \alpha-\beta\right)  =-\infty.$ Correspondingly,
$\bar{\theta}_{\min}=\bar{\theta}_{1}=\theta_{1}-\beta$ reaches its local
minimum. By the strong minimum principle, $\bar{u}_{\bar{1}\bar{1}}\equiv
\tan\bar{\delta}\equiv\bar{\lambda}_{1},$ under change of $\bar x$-coordinates if
necessary. Then the new Lagrangian graph (\ref{b-roration}) becomes%
\begin{gather*}
\left(  cx_{1}+su_{1}\left(  x\right)  ,cx^{\prime}+sD^{\prime}u\left(
x\right)  ,-sx_{1}+cu_{1}\left(  x\right)  ,-sx^{\prime}+cD^{\prime}u\left(
x\right)  \right) \\
=\left(  \bar{x}_{1},\bar{x}^{\prime},\tan\bar{\delta}\ \bar{x}_{1},D^{\prime
}\bar{u}\left(  \bar{x}^{\prime}\right)  \right)
\end{gather*}
Accordingly, $u_{1}\left(  x\right)  =\tan\left(  \bar{\delta}+\beta\right)
x_{1},$ hence $u_{11}\equiv\lambda_{1}$ is constant.

\smallskip
\noindent
\textbf{Remark.} The Legendre transform $u^{\ast}$ of strongly convex $u$ in
case $\Theta\in\left(  \pi,n\pi/2\right)  $ is obtained by performing the
rotation \eqref{b-roration} with $\beta=\pi/2$ followed by a reflection. In
fact, $\bar{u}=-u^{\ast}$ satisfies $\sum_{i=1}^{n}\arctan\bar{\lambda}%
_{i}\left(  D^{2}\bar{u}\right)  =\Theta-n\pi/2$ with $\bar{\lambda}%
_{i}\left(  D^{2}\bar{u}\right)<0.$ This equation's level set is concave on the part with all
$\bar{\lambda}_{i}(D^{2}\bar{u})  <0 $ as a graph along the
gradient $\nabla_{\bar{\lambda}}\sum_{i=1}^{n}\arctan\bar{\lambda}_{i}\ $. The
above argument for case $\Theta\in(\left(  2-n\right)  \pi/2,\pi]$ also leads
to the conclusions in the case $\Theta\in\left(  \pi,n\pi/2\right)  .$

\smallskip 
Next we consider the borderline situation, the attained local minimum $\arctan \lambda_1 = \alpha=\left(  \Theta-\pi\right)  /n.$ The previous new
representation by rotation is no longer valid, because $\bar{\theta}%
_{1}=-\pi/2$, or equivalently $\bar{\lambda}_{1}=-\infty.$ Whenever this rotation is possible, that is, when $\arctan\lambda
_{1}>\alpha,$ we have%
\begin{gather*}
\bar{\lambda}_{i}=\tan\left(  \theta_{i}-\beta\right)  =-a-\frac{1+a^{2}%
}{\lambda_{i}-a}\ \ \text{for }\tan\beta=-\tan^{-1}\alpha=-a^{-1}\\
\text{or\ \ \ \ \ }D^{2}\bar{u}=-aI-\left(  1+a^{2}\right)  \left(
D^{2}u-aI\right)  ^{-1}.
\end{gather*}
By the above mentioned concavity of the level set of $\bar{F}\left(  D^{2}%
\bar{u}\right)  =\left(  2-n\right)  \pi/2,$ which is preserved
under translation and scaling, the level set of%
\[
\tilde{F}\left(  A\right)  \overset{\text{def}}{=}\bar{F}\left(  -aI-\left(
1+a^{2}\right)  A\right)  =\left(  2-n\right)  \frac{\pi}{2}%
\]
is also concave as a graph along $\nabla_{A}\tilde{F}.$ Hence the 
condition for constant rank in [CGM, (2.2)] with $A=D^{2}u-aI$ and $X=D^{2}u_{e}$ holds.
Effectively, the condition \eqref{inverseconvexitycondition} for constant rank with $F\left(  A\right)  $ and
$u\left(  x\right)  $ replaced by $\tilde{F}\left(  A^{-1}\right)  $ and
$u\left(  x\right)  -a\left\vert x\right\vert ^{2}/2$ respectively holds.
Therefore, by \cite[Theorem 1.1]{CGM} or Theorem 4.1, $D^{2}u-aI$ has constant rank,
and in turn, $\lambda_{\min}=\lambda_{1}\equiv\tan\alpha=a.$
\end{proof}

The Hessian splitting conclusion of Theorem \ref{CrankSlag} is the main tool used to prove Corollary \ref{2homogeneous}. The idea is to split off a quadratic function in one variable from $u$ and then observe that the remaining part of the solution also satisfies the hypotheses of Theorem \ref{CrankSlag}. 

\begin{proof}[Proof of Corollary \ref{2homogeneous}]
We prove the quadratic rigidity under the assumption $ \arctan \lambda_{\min}  >  ( \Theta - \pi ) /n$. This result applied to the solution $v =-u $ of the equation $\sum_{i=1}^n \arctan \lambda_i (D^2 v ) = - \Theta$ yields the same rigidity when $ \arctan \lambda_{\max} (D^2 u ) < ( \Theta + \pi ) /n.$

Since $u$ is homogeneous of degree 2, $D^2 u $ is homogeneous of degree 0, so $\arctan \lambda_{\min} $ achieves its minimum value, $\arctan \lambda_{\min} \geq \arctan m_1 > ( \Theta - \pi ) /n$. By Theorem \ref{CrankSlag}, $\lambda_{\min} $ is constant and the Hessian $D^2 u $ splits such that, after a change of coordinates if necessary, $u_{11} \equiv m_1.$ 
Then, 
\begin{equation*}
u = \frac 12 m_1 x_1 ^2 + \tilde u ( x_2 , \dots, x_n ). 
\end{equation*} 
The function $\tilde u$ satisfies the equation 
\begin{equation} 
\sum_{i=1} ^ {n-1} \arctan \lambda_i ( D^2 \tilde u ) = \Theta - \arctan m_1 ,
\end{equation}
and $\arctan \lambda_i ( D^2 \tilde u ) \geq \arctan m_1  > ( \Theta - \pi ) /n $ for all $1 \leq i \leq n-1$. 

By Theorem \ref{CrankSlag} in dimension $n-1$, if $\arctan \lambda_{\min} ( D^2 \tilde u ) > ( \tilde \Theta  - \pi ) / ( n-1 ) , $ with $\tilde \Theta = \Theta - \arctan m_1$, then $\lambda_{\min } ( D^2 \tilde u) $ is constant and the Hessian $D^2 \tilde u$ splits. Indeed, by the assumption $\arctan m_1 >  ( \Theta - \pi ) / n$, or equivalently 
$$ \frac{ \Theta - \pi }{ n} > \frac{ \Theta - \arctan m_1 - \pi }{ n-1}, 
$$
we have 
$$ \lambda_{\min} ( D^2 \tilde u ) \geq \arctan m_1  > \frac{\Theta - \pi }{ n } > \frac{ \tilde \Theta - \pi }{n-1}.$$

Inductively, we conclude that $u$ is a quadratic polynomial. \end{proof} 

\noindent \textbf{Remark.} The assumption that $u$ is homogeneous degree 2 is used to ensure that every $\lambda_i$ attains an interior local minimum. In fact, if $u$ is a smooth solution of \eqref{EsLag} with $\arctan \lambda_{1} \geq ( \Theta - \pi ) / n , $ and every $\lambda_i$ for $1 \leq i \leq k$, attains an interior local minimum, then $\lambda_i$ is constant for $1 \leq i \leq k$. If all the local minima $\arctan \lambda_i > ( \Theta - \pi ) / n $, then the Hessian splits as $\lambda_i \equiv u_{ii} \equiv const.$ for $1 \leq i \leq k.$ 

We now proceed with the proof of Theorem \ref{rigidity}, which is an application of the Evans-Krylov theorem to an equation obtained by a rotation of the form \eqref{b-roration}.
\begin{proof}
[Proof of Theorem \ref{rigidity}.] We show the quadratic rigidity under the condition
$\arctan\lambda_{\min}\geq\gamma>\left(  \Theta-\pi\right)  /n.$ The same
rigidity for the solution $v=-u$ to the equation $\sum\arctan\lambda_{i}\left(
D^{2}v\right)  =-\Theta$ under the condition $\arctan\lambda_{\min}\left(
D^{2}v\right)  \geq-\gamma>\left(  -\Theta-\pi\right)  /n$ or equivalently $\arctan
\lambda_{\max}\left(  D^{2}u\right)  \leq\gamma<\left(  \Theta+\pi\right)  /n$
yields the corresponding conclusion under the $\lambda_{\max}$ condition.

The result is known if $\Theta\geq\pi$ as then $\lambda_{\min}\geq0$
\cite[Theorem 1.1]{Y1}, or if $\Theta\geq\pi-n\pi/6-\varepsilon^{\prime}\left(
n\right)  $ as then $\lambda_{\min}\geq-1/\sqrt{3}-\varepsilon\left(
n\right)  $ \cite[p.924]{WY} \cite[p.22]{LLY}, or if $\Theta<\left(  2-n\right)  \pi/2$
\cite[Theorem 1.1]{Y2}. We consider more than the remaining range, $\Theta\in
\lbrack\left(  2-n\right)  \pi/2,\pi),$ where our unified argument is in the
spirit of \cite{Y2}.

First, each Lagrangian angle $\theta_{i}=\arctan\lambda_{i}$ satisfies%
\[
\frac{\pi}{2}>\theta_{i}\geq\gamma>\frac{\Theta-\pi}{n}=\alpha\in\left[
-\frac{\pi}{2},0\right)  ,
\]
where we may and do assume $\gamma\leq0.$ In order to include the negative critical phase, set $\delta = ( \gamma - \alpha ) /2 > 0.$ 
Then the new representation (\ref{b-roration}) with $\beta$ replaced by $\tilde \beta =\pi/2 + \alpha + \delta  \in ( \delta , \pi/2 ) $ is valid, at least locally, such that%
$$
\bar{\theta}_{i}   =\theta_{i}-\frac{\pi}{2}-\alpha - \delta \in \left(
-\frac{\pi}{2}+ \delta ,\frac \pi 2 - \delta  \right) $$
and 
\begin{equation} 
\bar{F}\left(  D^{2}\bar{u}\right)     =\sum_{i=1}^{n}\arctan\bar{\lambda
}_{i}=\left(  2-n\right)  \frac{\pi}{2} - n \delta . \label{EsLagC-}%
\end{equation}
The new representation (\ref{b-roration}) of the original entire special
Lagrangian graph $\left(  x,Du\left(  x\right)  \right)  $ is in fact valid
for all $\bar{x}\in\mathbb{R}^{n}.$ This is because of the following distance
expansion property%
\begin{align*}
\left\vert \bar{x}^{2}-\bar{x}^{1}\right\vert  &  =\tilde{c}^{2}\left\vert
\left(  x^{2}-x^{1}\right)  +\tilde{t}\left[  Du\left(  x^{2}\right)
-Du\left(  x^{1}\right)  \right]  \right\vert ^{2}\\
&  =\tilde{c}^{2}\left\vert \left(  1-\tilde{t}m\right)  \left(  x^{2}%
-x^{1}\right)  +\tilde{t}\left[  Du\left(  x^{2}\right)  +mx^{2}-Du\left(
x^{1}\right)  -mx^{1}\right]  \right\vert ^{2}\\
&  \geq\tilde{c}^{2}\left\vert \left(  1-\tilde{t}m\right)  \left(
x^{2}-x^{1}\right)  \right\vert ^{2}+\tilde{t}^{2}\left\vert Du\left(
x^{2}\right)  +mx^{2}-Du\left(  x^{1}\right)  -mx^{1}\right\vert ^{2}\\
&  \geq\tilde{c}^{2}\left\vert \left(  1-\frac{\tan\left\vert \gamma
\right\vert }{\tan | \gamma + \delta |}\right)  \left(  x^{2}-x^{1}\right)  \right\vert ^{2},
\end{align*}
where $\left(  \tilde{c},\tilde{s}\right)  =\left(  \cos\left(  \frac{\pi}%
{2}+\alpha+\delta\right)  ,\sin\left(  \frac{\pi}{2}+\alpha+\delta\right)
\right)  ,$ $\tilde{t}=\tilde{s}/\tilde{c},$ $m=\tan\left\vert \gamma
\right\vert ,$ and convexity of $u\left(  x\right)  +m\left\vert x\right\vert
^{2}/2$ is used in the first inequality.

Finally, we have an entire solution $\bar{u}$ with bounded Hessian on the concave
level set \cite[Lemma 2.1]{Y2} of the now uniformly elliptic equation (\ref{EsLagC-}).
Scaling Evans-Krylov's H\"{o}lder seminorm estimate for the Hessian $\ D^{2}%
\bar{u}$ and then taking limit yields the constancy of \ $D^{2}\bar{u}.$ Thus
$\left(  \bar{x},D\bar{u}\left(  \bar{x}\right)  \right)  $, or equivalently $\left(
x,Du\left(  x\right)  \right),  $ is a plane, and in turn, the original
potential $u$ is a quadratic polynomial.
\end{proof}

\section{Proof of Theorem \ref{Cranksigma2}}

\begin{proof}
[Proof of Theorem \ref{Cranksigma2}.]First we assume the attained local minimum is strictly larger than the assumed lower bound,
 $\lambda_{\min}>-\sqrt{2/\left[  n\left(  n-1\right)  \right]
}\overset{\text{def}}{=}m.$ Let $w\left(  y\right)  $ be the Legendre-Lewy
transform, $(u\left(  x\right)  +m\left\vert x\right\vert ^{2}/2)^\ast.$ In terms
of gradient graphs in $\mathbb{R}^{n}\times\mathbb{R}^{n},$%
\begin{equation}
\left(  x,Du\left(  x\right)  +mx\right)  =\left(  Dw\left(  y\right)
,y\right)  . \label{LLLgraph}%
\end{equation}
A tangent plane of $\left(  x,Du\left(  x\right)  +mx\right)  $ is formed by
tangent vectors $\left(  e,\left(  D^{2}u\left(  x\right)  +mI\right)
e\right)  $ with $e\in\mathbb{R}^{n};$ correspondingly, tangent vectors
$\left(  \left(  D^{2}u\left(  x\right)  +mI\right)  ^{-1}e,e\right)  $ form
the tangent plane of $\left(  Dw\left(  y\right)  ,y\right)  $ at $y=Du\left(
x\right)  +mx.$ Thus%
$$
D^{2}w\left(  y\right)     =\left(  D^{2}u\left(  x\right)  +mI\right)
^{-1}, $$ and in terms of eigenvalues,
$$\mu_{i}    =\left(  \lambda_{i}+m\right)  ^{-1}>0.
$$
It follows that $\mu$ satisfies%
$$
1=\sum_{1\leq i<j\leq n}(  \mu_{i}^{-1}-m)  (  \mu_{j}%
^{-1}-m)  =\frac{\sigma_{n-2}\left(  \mu\right)  }{\sigma_{n}\left(
\mu\right)  }-\left(  n-1\right)  m\frac{\sigma_{n-1}\left(  \mu\right)
}{\sigma_{n}\left(  \mu\right)  }+1$$
or \begin{equation} \Lambda\left(  D^{2}w\right)  =\Lambda\left(  \mu\right)
\overset{\text{def}}{=}\frac{\sigma_{n-1}\left(  \mu\right)  }{\sigma
_{n-2}\left(  \mu\right)  }=\frac{1}{\left(  n-1\right)  m}, \label{Eqc}%
\end{equation}
where $m=\sqrt{2/\left[  n\left(  n-1\right)  \right]  }$ is used. The
function $\Lambda\left(  \mu\right)  $ is concave and $\nabla_{\mu}%
\Lambda\left(  \mu\right)  >0$ componentwise [L, Theorem 15.18], therefore the
level set $\Lambda\left(  \mu\right)  =1/\left[  \left(  n-1\right)
m\right]  $ is convex as a graph along $\nabla_{\mu}\Lambda.$ Hence%
\[
\bigtriangleup_{\Lambda}w_{ee}=\Lambda_{ij}\partial_{ij}w_{ee}\geq0.
\]
Consequently, a strong maximum principle holds for $w_{ee}.$

Now $\mu_{\max}=\left(  \lambda_{\min}+m\right)  ^{-1}$ reaches its local
maximum. By the strong maximum principle, $w_{11}\equiv l\equiv\mu_{\max}$, under change of $y$-coordinates if necessary. Then the gradient graph
(\ref{LLLgraph}) reads%
\[
\left(  x_{1},x^{\prime},u_{1}(  x)  +mx_{1},D^{\prime}u(
x)  +mx^{\prime}\right)  =\left(  ly_{1},D^{\prime}w(  y^{\prime
})  ,y_{1},y^{\prime}\right)  .
\]
Accordingly, $u_{1}(  x)  =\left(  l^{-1}-m\right)  x_{1},$ hence
$u_{11}\equiv\lambda_{\min}$ is constant.

Next we consider the borderline situation, the attained local minimum of $\lambda _ { \min } = m$. The previous
Legendre-Lewy transformation is no longer valid. Whenever the transformation
is possible, that is, when $\lambda_{\min}(  D^{2}u)  >m,$ we have
(\ref{Eqc}). By the above proved convexity of the level set of $\Lambda(
D^{2}w)  =1/\left[  \left(  n-1\right)  m\right]  $ as a graph along
$\nabla_{D^{2}w}\Lambda,$ the level set of
\[
\tilde{F}(  D^{2}u)  =-\Lambda\left(  \left(  D^{2}u+mI\right)
^{-1}\right)  =-1/\left[  \left(  n-1\right)  m\right]
\]
is concave as a graph along $\nabla_{D^{2}u}\tilde{F}\left(  D^{2}u\right).$ As in the
proof of Theorem \ref{CrankSlag}, the condition for constant rank in \cite[(2.2)]{CGM} with
$A=D^{2}u+mI$ and $X=D^{2}u_{e}$ holds. Effectively, the condition \eqref{inverseconvexitycondition} for constant rank with $F\left(  A\right)  $ and $u\left(  x\right)  $ replaced
by $\tilde{F}\left(  A^{-1}\right)  $ and $u\left(  x\right)  +m\left\vert
x\right\vert ^{2}/2$ respectively holds. Therefore, by \cite[Theorem 1.1]{CGM} or
Theorem 4.1, $D^{2}u+mI$ has constant rank, and in turn, $\lambda_{\min
}=\lambda_{1}$ is constant.
\end{proof}

\section{Appendix} \label{sect4} 
In this section we present a viscosity approach to prove the following constant rank theorem for nonlinear elliptic Hessian equations. As mentioned in the introduction, a similar viscosity proof appeared in \cite{BIS}. For completeness, we include our proof here. For convenience, we state and prove the theorem under the assumption that $F(D^2u ) = f ( \lambda_1 , \dots , \lambda_n )$ where $f$ is a symmetric function of its arguments, but the proof can be adapted to accommodate a general Hessian equation. 
\begin{theorem} \label{constantrank}
Assume $u \in C^3$ is a convex solution of the elliptic equation $F(D^2u) =0 $ where $F(M^{-1})$ is a convex function on $\sym^+(n)$. Then $D^2 u $ has constant rank. \end{theorem} 
One of the obstacles in proving a constant rank theorem is that the eigenvalues of the Hessian matrix need not be differentiable. Previous works deal with this issue by computing with expressions in terms of the elementary symmetric polynomials of the eigenvalues, which are differentiable \cite{CGM, BG}, or by approximating the solution $u$ by a polynomial to ensure that points where the eigenvalues may not be differentiable are confined to an analytic variety \cite{SW1}. The current approach is to compute directly with the eigenvalues of the Hessian in order to derive differential inequalities for the eigenvalues which hold in the viscosity sense. In particular, this approach gives a simple argument for establishing the constancy of the smallest eigenvalue. 

We begin with a lemma establishing a formula for 1-sided derivatives of eigenvalues at matrices in $\sym(n)$ with repeated eigenvalues.  

\begin{lemma} \label{onesided}
Assume $M= \diag(\lambda_1, \dots, \lambda_n) $ with $\lambda_1 \leq \dots \leq \lambda_n $. Furthermore, assume $\lambda_{j-1} < \lambda_j = \dots = \lambda_k< \lambda_{k+1} $. Then, if $j \leq i \leq k$ and $A \in \sym(n)$, then
$$
\frac{d^+}{dt} \bigg \vert _{t=0} \lambda_i ( M+ t A ) = \lim_{t \to 0^+} \frac{ \lambda_i ( M + t A ) - \lambda_i }{t } 
$$
exists and 
\begin{equation*} \frac{d^+}{dt} \bigg \vert _{t=0} \lambda_i ( M+ t A ) = \lambda_{i-j+1} ( A|_{E} ),
\end{equation*} 
where $E= \spn(e_j, \dots, e_k )$ is the $\lambda_i$-eigenspace of $M$ and $A|_E$ is the $(k-j+1)\times (k-j+1)$ matrix with entries $(A_{rs})_{r,s=j}^k.$ 
\end{lemma} 

\begin{proof} We may choose an orthogonal basis for $E$ to diagonalize $A|_E$. Once $A|_E$ is diagonal, it is clear that for small nonzero $t$, the eigenvalues of the perturbed matrix $(M + t A)|_E$ satisfy 
\begin{equation*} 
\frac{ d^+ }{dt} \bigg \vert_{t=0} \lambda_{i-j+1} (( M+ tA)|_E) = \lambda_{i-j+1} ( A | _{E} ). 
\end{equation*} 
The entries of $tA$ outside of the block $tA|_E$ make contributions on the order of $t^2$ to $\lambda_i(M+tA)$, so this proves the lemma. 
\end{proof}

A consequence of Lemma \ref{onesided} is that if $M$ has repeated eigenvalues $\lambda_j = \dots = \lambda_k$, then $\lambda_{i}$ ($j \leq i \leq k)$ is differentiable in the direction $A$ at $M$ if and only if $\lambda_{i-j+1} (A|_E) = \lambda_{k-i+1}(A|_E)$. To see this, note that if 
$\frac{d}{dt}  \vert_{t=0} \lambda_i (M+ tA ) $ exists, then 
\begin{equation*} 
\begin{aligned} 
 \frac{d^+}{dt} \bigg \vert _{t=0} \lambda_i (M+ t A )&= \frac{d^-}{dt} \bigg \vert _{t=0} \lambda_i (M+ t A )\\ 
 &= \lim_{ t\to 0^-} \frac{ \lambda_i ( M + tA ) - \lambda_im }{t}\\
 &= - \lim_{t \to 0^+} \frac{ \lambda_i ( M -tA ) -\lambda_im }{ t}\\
 & =- \frac{d^+}{dt} \bigg \vert _{t=0} \lambda_i (M- t A )\\
 &= -\lambda_{i-j+1}  ( - A |_E) .\end{aligned} \end{equation*} 
Note that $- \lambda_{i-j+1} ( -A|_E) = \lambda_{k-i+1} (A|_E) .$\par 
In particular, $\lambda_j$ is differentiable in the direction $A$ if and only if $A|_E$ is a scalar matrix.

\begin{proof}[Proof of Theorem \ref{constantrank}]

Recall that $\lambda_j$ is a Lipschitz function, $C^1$ on the set $\{ p \: : \: \lambda_{j-1}  < \lambda_{j}  < \lambda_{j+1} \}.$ We assume that $0=\lambda_1(0)=\dots =\lambda_{n-r}(0),$ so that the rank of $D^2u$ reaches its minimum at $0$. 
All computations will be performed at a fixed point $p$, assuming that $D^2 u (p) $ is diagonal with $\lambda_j(p) = u_{jj}(p).$ Note that the assumption that $F$ is a symmetric function of the eigenvalues implies that $F_{ij} (D^2 u (p)) = f_{\lambda_i } \delta_{ij} .$ 
For a positive integer $t \leq n$, let $$I_t= \{ (i,j,k,l) \: : \: 1 \leq i,j,k,l \leq n, \text{ at least one of } i,j,k,l \leq t \}.$$

Step 1: Superharmonicity of $\lambda_1$. 

Assume first that $\lambda_1(p) < \lambda_2(p).$ 
At $p$, we have
 \begin{align}  \partial_i \lambda _1 & = \partial _i u_{11} \label{lambda1gradient}, \\ \partial_{ii} \lambda_1&= \partial_{ii} u_{11} + \sum_{ j > 1 } \frac{ 2}{ \lambda_1 - \lambda_j }( \partial_i u_{1j} )^2 \notag. \end{align}
Differentiating the equation $F(D^2u ) = 0 $, we obtain 
\begin{equation*}  0 = \partial_{11} F(D^2u ) = \sum_{i,j,k,l } F_{ij,kl} u_{ij1} u_{kl1} + \sum_{i,j} F_{ij} u_{ij11}, \end{equation*} 
so 
\begin{equation*} \begin{aligned} 
 \triangle _F u_{11} &= -\sum_{i,j,k,l} F_{ij,kl} u_{ij1} u_{kl1} \\
 & = - \sum_{i,j,k,l>1} F_{ij,kl} u_{ij1} u_{kl1} - \sum_{(i,j,k,l)\in I_1} F_{ij,kl} u_{ij1} u_{kl1}. \\
  \end{aligned} \end{equation*} 
  where $\triangle_F=\sum_{i,j} F_{ij} \partial_{ij} $.
  
  By the inverse-convexity we have
 \begin{equation}\label{inverseconvexitycondition}  - \sum_{i,j,k,l>1} F_{ij,kl} u_{ij1} u_{kl1} \leq 2 \sum_{i,j>1} F_{ii} \lambda_j^{-1}  u_{ij1}^2. \end{equation} 
 The sum of terms with at least one index equal to 1 is of the form $B \cdot D u_{11} $, and at $p$, $Du_{11} = D \lambda_1$ by \eqref{lambda1gradient}. 
 Therefore, at $p$
 $$\triangle _F u_{11} \leq 2 \sum_{i,j>1} F_{ii} \lambda_j^{-1} u_{ij1}^2 + B \cdot D \lambda_1. $$
 Then, 
 \begin{equation}\label{l1superharmonic} \begin{aligned} 
 \triangle_F \lambda_1 & = \triangle _F u_{11} +2\sum_{i} \sum_{j > 1  } F_{ii} \frac{ 1}{ \lambda_1 - \lambda_j } u_{ij1}^2 \\
& \leq B\cdot D \lambda_1 + 2 \sum_{i,j>1} F_{ii} \left ( \frac{1}{\lambda_j}+ \frac{ 1} { \lambda_1 - \lambda _j } \right )  u_{ij1}^2+ 2 \sum_{j>1} F_{11} \frac{ 1}{\lambda_1 - \lambda_j } u_{1j1}^2 \\
&\leq B \cdot D \lambda_1
 \end{aligned} \end{equation} 
 because $\frac{ 1}{ \lambda _j } - \frac1{ \lambda_j - \lambda_1 } \leq 0.$ 
 \par 
Now assume that 
 $\lambda_1(p) = \dots = \lambda_s(p) < \lambda_{s+1}(p).$
Suppose $\phi$ is a smooth function supporting $\lambda_1$ from below at $p$. 
 Then 
 \begin{equation*}  \phi \leq \lambda_1  \leq \frac 1 s \sum_{m\leq s}  \lambda_m  \end{equation*} 
  with equality at $p$. 
Observe that $\sum_{m\leq s}  \lambda_m$ is differentiable in a neighborhood of $p$. It follows from the assumption that $\phi$ supports $\lambda_1$ from below that $\lambda_1$ is also differentiable at $p$. Therefore Lemma \ref{onesided} implies that $u_{ije}(p)  =0$ if $i \neq j$ for $i,j \leq s$. 

We now compute 
$ \triangle _F \sum_{m\leq s}  \lambda_m.$ We have 
 \begin{align*} 
\triangle_F \sum_{m\leq s}  \lambda_m & = \triangle_F \sum_{m\leq s}  u_{mm} + 2 \sum_{i} \sum_{j>s}\sum_{m\leq s}  F_{ii} \frac{ u_{ijm}^2 } { \lambda_m - \lambda_j} \\
& = - \sum_{i,j,k,l} \sum_{m\leq s}  F_{ij,kl} u_{ijm} u_{klm} + 2 \sum_{i} \sum_{j>s}\sum_{m\leq s}  F_{ii} \frac{ u_{ijm}^2 } { \lambda_m - \lambda_j}
\\   & = - \sum_{i,j,k,l>s} \sum_{m\leq s}  F_{ij,kl} u_{ijm} u_{klm}  +  2 \sum_{i} \sum_{j>s}\sum_{m\leq s}  F_{ii} \frac{ u_{ijm}^2 } { \lambda_m - \lambda_j}  \\ & \quad- \sum_{(i,j,k,l) \in I_s }  \sum_{m\leq s}  F_{ij,kl} u_{ijm} u_{klm} 
\\& \leq 2 \sum_{i,j>s} \sum_{m \leq s} F_{ii} \frac{ u_{ijm}^2}{ \lambda_j } + 2  \sum_{i} \sum_{j>s}\sum_{m\leq s}  F_{ii} \frac{ u_{ijm}^2 } { \lambda_m - \lambda_j} \\ & \quad + \sum_{m \leq s } B_m \cdot D \lambda_m .
\end{align*} 
Lemma \ref{onesided} also implies that $D \lambda_i = D \lambda_j $ for $i,j \leq s$, so 
\begin{equation} \label{l1viscosity} \triangle _F \sum_{m\leq s}  \lambda_m  \leq   B \cdot D \lambda_1. \end{equation}  
Combining \eqref{l1superharmonic} with \eqref{l1viscosity} shows that $\lambda_1$ is a viscosity supersolution of the equation 
$ \triangle_F w - B \cdot D w =0, $
so by the strong minimum principle and the assumption that $\lambda_1 (0) =0$ it follows that $\lambda_1 \equiv 0.$ 
\par 
Step 2: Superharmonicity of $\lambda_{a+1} $ given $\lambda_1 = \dots \lambda_a =0.$ 

Assume that for some $a<r$, $\lambda_a \equiv 0.$ Fix $p \in \Omega$ and assume that $\lambda_{a+1} (p) = \dots = \lambda_{a+s} (p) < \lambda_{a+s+1}(p). $ Let $\mu = \sum_{a<m \leq a+ s } \lambda_{m}. $

If $\lambda_{a+1} (p) > 0$, then we compute 
 \begin{equation*}\begin{aligned} 
  \triangle_F \mu & =  \triangle_F\mu + 2\triangle_F \sum_{m \leq a } \lambda_m \\ 
& =  \triangle_F \sum_{a<m \leq a+s } u_{mm} + 2 \sum_i \sum_{j>a+s} \sum_{a<m \leq a+s } F_{ii} \frac{ u^2_{ijm}}{\lambda_{m} -\lambda_j } \\ 
& \quad + 2  \sum_i \sum_{j \leq a} \sum_{a < m \leq a +s } F_{ii} \frac{ u^2_{ijm}}{\lambda_{m} } + 2 \triangle_F \sum_{m\leq a } u_{mm} \\ 
& \quad+ 4 \sum_{i} \sum_{j > a} \sum_{m  \leq a} F_{ii} \frac{u^2_{ijm}}{ - \lambda_j}  \\ \end{aligned}\end{equation*} \begin{equation*}\begin{aligned}
& = - \sum_{i,j,k,l} \sum_{a < m \leq a +s } F_{ij,kl} u_{ijm } u_{klm} - 2 \sum_{i,j,k,l} \sum_{m \leq a } F_{ij,kl} u_{ijm} u_{klm}\\ 
& \quad + 2 \sum_i \sum_{j>a+s} \sum_{a < m \leq a +s } F_{ii} \frac{ u^2_{ijm}}{\lambda_{m} -\lambda_j }+ 2  \sum_i \sum_{j \leq a} \sum_{a < m \leq a +s } F_{ii} \frac{ u^2_{ijm}}{\lambda_{m}  } \\ 
& \quad+ 4 \sum_{i} \sum_{j>a} \sum_{m \leq a} F_{ii} \frac{u^2_{ijm}}{ - \lambda_j}  \\ 
\end{aligned}\end{equation*} 

By the inverse-convexity assumption, 
\begin{equation*}\begin{aligned} 
- \sum_{i,j,k,l} \sum_{a < m \leq a +s } F_{ij,kl} u_{ijm } u_{klm}&  \leq  2 \sum_{i,j>a+s} \sum_{a < m \leq a +s } F_{ii} \frac{1}{\lambda_j} u_{ijm}^2
\\ & \quad - \sum_{(i,j,k,l) \in I_{a+s} } \sum_{a < m \leq a +s }F_{ij,kl}  u_{ijm} u_{klm} 
\end{aligned} \end{equation*} 

and 

\begin{equation*} \begin{aligned} 
- 2 \sum_{i,j,k,l} \sum_{m \leq a } F_{ij,kl} u_{ijm} u_{klm} &\leq 4 \sum_{i,j> a+s }\sum_{m \leq a}  F_{ii} \frac 1 {\lambda_j} u_{ijm}^2 
\\ & \quad - 2 \sum_{(i,j,k,l) \in I_{a+s}} \sum_{m \leq a} F_{ij,kl} u_{ijm} u_{klm}.
\end{aligned} \end{equation*} 

Therefore 
\begin{equation} \begin{aligned} 
\triangle_F \mu & \leq  2 \sum_{i,j>a+s} \sum_{a < m \leq a +s } F_{ii} \frac{1}{\lambda_j} u_{ijm}^2
- \sum_{(i,j,k,l) \in I_{a+s} } \sum_{a < m \leq a +s }F_{ij,kl}  u_{ijm} u_{klm}   \\  
& \quad 
+ 4 \sum_{i,j> a+s }\sum_{m \leq a}  F_{ii} \frac 1 {\lambda_j} u_{ijm}^2 
- 2 \sum_{(i,j,k,l) \in I_{a+s}} \sum_{m \leq a} F_{ij,kl} u_{ijm} u_{klm} 
\\
  & \quad
+2 \sum_i \sum_{j>a+s} \sum_{a < m \leq a +s } F_{ii} \frac{ u^2_{ijm}}{\lambda_{m} -\lambda_j } + 2  \sum_i \sum_{j \leq a} \sum_{a < m \leq a +s } F_{ii} \frac{ u^2_{ijm}}{\lambda_{m} } \\  
& \quad+ 4 \sum_{i} \sum_{j > a} \sum_{m  \leq a} F_{ii} \frac{u^2_{ijm}}{- \lambda_j}\\
   & = 2 \sum_{i,j>a+s} \sum_{a < m \leq a +s } F_{ii} \frac{1}{\lambda_j} u_{ijm}^2
- \sum_{(i,j,k,l) \in I_{a+s}} \sum_{a < m \leq a +s }F_{ij,kl}  u_{ijm} u_{klm}  
 \\  
 & \quad 
+ 4 \sum_{i,j> a+s }\sum_{m \leq a}  F_{ii} \frac 1 {\lambda_j} u_{ijm}^2 
- 2 \sum_{(i,j,k,l) \in I_{a+s}} \sum_{m \leq a} F_{ij,kl} u_{ijm} u_{klm} 
\\ \end{aligned}\end{equation} \begin{equation*}\begin{aligned} & \quad 
+2 \sum_i \sum_{j>a+s} \sum_{a < m \leq a +s } F_{ii} \frac{ u^2_{ijm}}{\lambda_{m} -\lambda_j } + 2  \sum_i \sum_{j \leq a} \sum_{a < m \leq a +s } F_{ii} \frac{ u^2_{ijm}}{\lambda_{m} } 
\\  
& \quad+ 4 \sum_{i} \sum_{a < j \leq a+s } \sum_{m \leq a} F_{ii} \frac{u^2_{ijm}}{ - \lambda_j}+4 \sum_{i} \sum_{j > a+s} \sum_{m  \leq a} F_{ii} \frac{u^2_{ijm}}{ - \lambda_j}
\\   & \leq - \sum_{(i,j,k,l) \in I_{a+s} } \sum_{a < m \leq a +s } F_{ij,kl} u_{ijm } u_{klm} - 2 \sum_{(i,j,k,l) \in I_{a+s}  } \sum_{m \leq a  } F_{ij,kl} u_{ijm} u_{klm} 
\\   & \quad + 2 \sum_i \sum_{ j > a } \sum_{m \leq a } F_{ii} \frac{ u_{ijm}^2 }{ - \lambda_j }, \label{mu} 
\end{aligned} \end{equation*} 

because 
\begin{equation*}
 \sum_{i,j>a+s} \sum_{a < m \leq a +s } F_{ii} \frac{1}{\lambda_j} u_{ijm}^2+  \sum_i \sum_{j>a+s} \sum_{a < m \leq a +s } F_{ii} \frac{ u^2_{ijm}}{\lambda_{m} -\lambda_j } \leq 0 
\end{equation*} 
and 
\begin{equation*} 
 \sum_{i,j> a+s }\sum_{m \leq a}  F_{ii} \frac 1 {\lambda_j} u_{ijm}^2 
+ \sum_{i} \sum_{j > a+s} \sum_{m  \leq a} F_{ii} \frac{u^2_{ijm}}{ - \lambda_j}
\leq 0 . 
\end{equation*}

We now estimate the terms $F_{ij,kl} u_{ijm} u_{klm}$. We use the same technique as \cite[p.6527]{SW1}. First, 
\begin{equation*}    - \sum_{(i,j,k,l) \in I_{a+s} } \sum_{a<m \leq a+s } F_{ij,kl} u_{ijm } u_{klm} - 2 \sum_{(i,j,k,l) \in I_{a+s}} \sum_{m \leq a  } F_{ij,kl} u_{ijm} u_{klm}  \end{equation*} 
\begin{align*}    & \leq C (\| F \| _{C^2 }, \| u \|_{C^3 }, n,a,s)\sum_i \sum_{j,k \leq a+s } | u_{ijk} |\\ &  \leq C \left ( \sum_i \sum_{j < k \leq a+s } | u_{ijk} |   + \sum_i \sum_{j \leq a+s }  | u_{ijj} | \right ).  \end{align*} 

For $j \leq a$, $\partial_i \lambda_j =0$. 
If $\phi$ is a smooth function supporting $\lambda_{a+1}$ from below at $p$, then $\lambda_{a+1}$ is differentiable at $p$ because $ \mu /s$ supports $\lambda_{a+1}$ from above. Therefore Lemma \ref{onesided} implies that $D \lambda_{a+1}(p) = \dots = D \lambda_{a+s} (p). $ We then have $u_{ijj}(p) = \partial_i \lambda_{a+1} (p) $ if $a< j \leq a +s$.

For the remaining terms, we use the Cauchy-Schwarz inequality. It follows from Lemma \ref{onesided} that the matrix $(\partial_{i} u_{jk} ) _{j,k} $ is 0 in the upper left $a\times a$ block and the $s \times s$ block $(\partial_{i} u_{jk} )_{j,k=a+1}^{a+s} $ is diagonal. Therefore the only nonzero terms of the sum $\sum_i \sum_{j<k\leq a+s } |u_{ijk}|$ are those with $j \leq a < k.$ By assumption, $\mu>0$ at $p$,  so
\begin{equation*} | u_{ijk} | \leq \frac{u_{ijk}^2}{M \mu} + M \mu  \end{equation*} 
where $M$ is a large constant to be chosen. By ellipticity, 
\begin{equation*} \frac{ u_{ijk}^2}{M \mu } \leq C  \sum_i F_{ii} \frac{ u_{ijk}^2 } { M \mu } \frac{  \mu } { \lambda_k - \lambda_j } . \end{equation*} 
It follows that 
\begin{equation} \sum_i \sum_{j\leq a} \sum_{a < k \leq a+s }| u_{ijk}| \leq \frac C M \sum_i \sum_{j\leq a} \sum_{a < k \leq a+s } F_{ii} \frac{ u_{ijk}^2 }{ \lambda_k - \lambda_j }. \label{CSestimate} \end{equation} 
Combining \eqref{CSestimate} with the result of \eqref{mu}, we have \begin{equation} \label{finalestimate} \begin{aligned} \triangle_F \mu &\leq 2 \sum_i \sum_{j>a} \sum_{m \leq a } F_{ii} \frac{u_{ijm}^2}{\lambda_m -\lambda_j} + \frac C M \sum_i \sum_{j\leq a} \sum_{a < k \leq a+s } F_{ii} \frac{ u_{ijk}^2 }{ \lambda_k - \lambda_j }
\\ & \quad + CM \mu + \sum_i \sum_{j \leq a+s} | u_{ijj} | 
\\&  =  2 \sum_i \sum_{j>a} \sum_{m \leq a } F_{ii} \frac{u_{ijm}^2}{\lambda_m -\lambda_j} + \frac C M \sum_i \sum_{j\leq a} \sum_{a < k \leq a+s } F_{ii} \frac{ u_{ijk}^2 }{ \lambda_k - \lambda_j }
\\ & \quad + CM \mu + B | D \mu | .
\end{aligned} 
\end{equation} 
Provided $M$ is large enough, the sum of the first two terms in \eqref{finalestimate} is nonpositive, so 
\begin{equation*}  \triangle_F \mu \leq CM \mu + B | D \mu |.\end{equation*} 
It then follows that 
$$ \triangle_F \lambda_{a+1} \leq C \lambda_{a+1} + B | D \lambda_{a+1}| $$ in the viscosity sense on the set $\{\lambda_{a+1}>0\}$. That the differential inequality holds also on the set $\{ \lambda_{a+1}=0\}$ follows from \eqref{l1viscosity}. \par 
Now the theorem follows by the strong minimum principle and induction on $a$. 
\end{proof}

\textbf{Acknowledgments. }{We are very grateful to Pengfei Guan for initiating this project and generously sharing his insights. WJO is supported by the NSF Graduate Research Fellowship Program under grant No. DGE-2140004. YY is partially
supported by NSF grant No. DMS-2054973.}


\begin{thebibliography}{9999}                                                                                             %


\bibitem[ALL]{ALL}Alvarez, O.; Lasry, J.-M.; Lions, P.-L. \emph{Convex
viscosity solutions and state constraints}, J. Math. Pures Appl. (9)
\textbf{76} (1997), no. 3, 265-288.


\bibitem[BG]{BG} Bian, Baojun; Guan, Pengfei. \emph{A microscopic convexity principle for nonlinear partial
              differential equations.} Invent. Math. \textbf{177} (2009), no. 2, 307--335. 

\bibitem[BIS]{BIS} Bryan, Paul; Ivaki, Mohammad N.; Scheuer, Julian. \emph{Constant rank theorems for curvature problems via a viscosity approach.}
Calc. Var. Partial Differential Equations \textbf{62} (2023), no. 3, Paper No. 98.

\bibitem[BS]{BS} Bhattacharya, Arunima; Shankar, Ravi. \emph{Optimal regularity for Lagrangian mean curvature type equations}. Preprint (2020), 	arXiv:2009.04613. 

\bibitem[CGM]{CGM}Caffarelli, Luis; Guan, Pengfei; Ma, Xi-Nan. \emph{A constant
rank theorem for solutions of fully nonlinear elliptic equations.} Comm. Pure
Appl. Math. \textbf{60} (2007), no. 12, 1769--1791.


\bibitem[CY]{CY}{Chang, Sun-Yung Alice; Yuan, Yu.}
 \emph{A {L}iouville problem for the sigma-2 equation.}
Discrete Contin. Dyn. Syst.
 \textbf{28}, no. 2 (2010), 659--664.


 \bibitem[KL]{KL}  {Korevaar, Nicholas J.; Lewis, John L.}
\emph{Convex solutions of certain elliptic equations have constant
 rank {H}essians}, {Arch. Rational Mech. Anal.}
\textbf{97},(1987), no. {1},{19--32}.

\bibitem[L]{L}{Lieberman, Gary M.}
\emph{Second order parabolic differential equations.}
{World Scientific Publishing Co., Inc., River Edge, NJ},
(2005).

\bibitem[Lc]{Lc} {Li, Caiyan.} \emph{ Non-polynomial entire solutions to Hessian equations.} Calc. Var. Partial Differential Equations \emph{60} (2021), no. 4, Paper No. 123, 6 pp.

\bibitem[LLY]{LLY}{Li, Dongsheng; Li, Zhisu; Yuan, Yu.} \emph{A Bernstein problem for special Lagrangian equations in exterior domains.} Adv. Math. \textbf{361} (2020), 106927, 29 pp.




\bibitem[MS]{MS} {Mooney, Connor; Savin, Ovidiu.} \emph{Non $C^1$ solutions to the special Lagrangian equation.} Duke Math. J., to appear.

\bibitem[NTV]{NTV} {Nadirashvili, Nikolai; Tkachev, Vladimir; Vl\u{a}du\c{t}, Serge.} \emph{A non-classical solution to a Hessian equation from Cartan isoparametric cubic.} Adv. Math. \textbf{231} (2012), no. 3-4, 1589--1597.

\bibitem[NV1]{NV1}Nadirashvili, Nikolai; Vl\u{a}du\c{t}, Serge. \emph{Singular
solution to special Lagrangian equations.} Ann. Inst. H. Poincar\'{e} Anal.
Non Lin\'{e}aire \textbf{27} (2010), no. 5, 1179--1188.

\bibitem[NV2]{NV2} {Nadirashvili, Nikolai; Vl\u{a}du\c{t}, Serge.} \emph{Homogeneous solutions of fully nonlinear elliptic equations in four dimensions.} Comm. Pure Appl. Math. \textbf{66} (2013), no. 10, 1653--1662.



\bibitem[SW1]{SW1} {Sz\'{e}kelyhidi, G\'{a}bor; Weinkove, Ben.} \emph{On a constant rank theorem for nonlinear elliptic {PDE}s}. Discrete Contin. Dyn. Syst. \textbf{36}, (2016), no. 11, 6523--6532. 

\bibitem[SW2]{SW2} {Sz\'{e}kelyhidi, G\'{a}bor; Weinkove, Ben.} \emph{Weak Harnack inequalities for eigenvalues and constant rank theorems.}.  Comm. Partial Differential Equations \textbf{46} (2021), no. 8, 1585–1600.

\bibitem[SY]{SY} {Shankar, Ravi; Yuan, Yu.} \emph{Rigidity for general semiconvex entire solutions to the sigma-2 equation.} Duke Math. J. \textbf{171} (2022), no. 15, 3201–3214.

\bibitem[W]{W} {Warren, Micah.} \emph{Nonpolynomial entire solutions to $\sigma_k$ equations.} Comm. Partial Differential Equations \textbf{41} (2016), no. 5, 848–853.

\bibitem[WdY]{WdY} {Wang, Dake; Yuan, Yu.} \emph{Singular solutions to special Lagrangian equations with subcritical phases and minimal surface systems.} Amer. J. Math. \textbf{135} (2013), no. 5, 1157–1177.

\bibitem[WY]{WY} {Warren, Micah; Yuan, Yu.} \emph{A Liouville type theorem for special Lagrangian equations with constraints.} Comm. Partial Differential Equations \textbf{33} (2008), no. 4-6, 922–932.


\bibitem[Y1]{Y1}Yuan, Yu. \emph{A Bernstein problem for special Lagrangian
equations.} Invent. Math. \textbf{150} (2002), 117--125.

\bibitem[Y2]{Y2}Yuan, Yu. \emph{Global solutions to special Lagrangian
equations.} Proc. Amer. Math. Soc. \textbf{134} (2006), no. 5, 1355--1358.


\end{thebibliography}
\end{document}